\documentclass[12pt]{article}
\usepackage{amsmath,amsfonts,amssymb,amsthm,latexsym,pb-diagram}
\newtheorem{thm}{Theorem}
\newtheorem{defn}[thm]{Definition}
\newtheorem{lem}[thm]{Lemma}
\newtheorem{cor}[thm]{Corollary}

\newtheorem{rem}[thm]{Remark}
\DeclareMathOperator{\modu}{mod}
\begin{document}
\begin{center}
\textbf{A REMARK ON RELATIVELY PRIME SETS}
\end{center}
\begin{center}
Prapanpong Pongsriiam\\
Department of Mathematics, Faculty of Science, Silpakorn University, Ratchamankanai Rd, Nakornpathom, Thailand, 73000.\\
 Email: prapanpong@gmail.com
\end{center}
\begin{center}
\textbf{Abstract}
\end{center}
Four functions counting the number of subsets of $\{1, 2, \ldots, n\}$ having particular properties are defined by Nathanson and generalized by many authors. They derive explicit formulas for all four functions. In this paper, we point out that we need to compute only one of them as the others will follow as a consequence. Moreover, our method is simpler and leads to more general results than those in the literature.\vspace{0.5cm}

\section{Introduction}

There are a number of articles concerning relatively prime subsets and Euler phi function for sets. Most of them show the calculation of explicit formulas for certain functions. Their main tools are the M$\ddot{\text{o}}$bius inversion formula and the inclusion-exclusion principle. In this paper, we give simpler and shorter calculations which lead to the results extending those in the literature. To be precise, we cover the results of Nathanson \cite{Na}, Nathanson and Orosz \cite{NO}, El Bachraoui \cite{Ba1}, \cite{Ba2}, \cite{Ba3}, \cite{Ba4}, \cite{Ba5}, El Bachraoui and Salim \cite{BS}, Ayad and Kihel \cite{AK1}, \cite{AK2}, and Shonhiwa \cite{Sh1}, \cite{Sh2}. We show how to apply our method to obtain all results mentioned above and their generalization. Now, let us introduce the following notations and definitions which will be used throughout this paper.\\ 

Unless stated otherwise, we let $a, b, k, m, n$ be positive integers, $\gcd(a, b)$ the greatest common divisor of $a$ and $b$, $[a,b] = \{ a, a+1, \ldots, b\}$, $A, X$ finite subsets of positive integers, $|A|$ the cardinality of the set $A$, $\gcd(A)$ the greatest common divisor of the elements of $A$, $\gcd(A,n)$ means $\gcd(A\cup\{n\})$, $\left\lfloor x\right\rfloor$ the greatest integer less than or equal to $x$, and $\mu$ the M$\ddot{\text{o}}$bius function.\\

\indent A nonempty finite subset $A$ of positive integers is said to be relatively prime if $\gcd(A) = 1$ and is said to be relatively prime to $n$ if $\gcd(A,n)=1$. The function counting the number of relatively prime subsets of $\{1, 2, \ldots, n\}$ and other related functions are defined by Nathanson \cite{Na} and generalized by many authors. We summarize them in the following definition.
\begin{defn}
Let $X$ be a nonempty finite subset of positive integers. Define $f(X)$ to be the number of relatively prime subsets of $X$, $f_k(X)$ the number of such subsets with cardinality $k$, $\Phi(X, n)$ the number of subsets $A$ of $X$ which is relatively prime to $n$ and $\Phi_k(X, n)$ the number of such subsets $A$ with $|A| = k$.
\end{defn}

Nathanson \cite{Na} first considered the case $X = [1,n] $. Using the M$\ddot{\text{o}}$bius inversion formula for functions of several variables, El Bachraoui \cite{Ba1}, \cite{Ba2}, Nathanson and Orosz \cite{NO} generalized the formulas to the case $X = [m, n]$ and $X = [1,m]$ (see details in the table). Then Ayad and Kihel \cite{AK1}, \cite{AK2} generalized all results mentioned above to the case of an arithmetic progression $X = \{a, a+b, \ldots, a+(m-1)b\}$ by using the inclusion-exclusion principle. In another direction, El Bachraoui and Salim \cite{BS} obtained the formulas for the case $X = [\ell_1,m_1]\cup[\ell_2,m_2]\cup \cdots \cup [\ell_k,m_k]$. In addition, Shonhiwa \cite{Sh1}, \cite{Sh2} and Toth \cite{To} gave results where various constraints are assumed. We summarize the development in the table below.
\begin{table}[h] 
\begin{center}
\begin{tabular}{|l|l|}
\hline
\quad\quad\quad \quad\quad Authors & Formulas for $f(X)$, $f_k(X)$, $\Phi(X, n)$, $\Phi_k(X, n)$\\\hline
Nathanson \cite{Na} & $X=[1,n]$\\\hline
El Bachraoui \cite{Ba1} &$X=[m,n]$\\
Nathanson and Orosz \cite{NO} &\\\hline
El Bachraoui \cite{Ba2} & $X=[1,m]$\\\hline
 &$X=\{a,a+b,a+(n-1)b\}$, \\
 Ayad and Kihel \cite{AK1}, \cite{AK2}&  $X=[\ell,m]$,\\
 &$X=\{a,a+b,a+(m-1)b\}$\\\hline
El Bachraoui \cite{Ba3}, \cite{Ba5} & $X = [1,m_1]\cup[\ell_2,m_2]$\\
El Bachraoui and Salim \cite{BS} & $X = [\ell_1,m_1]\cup[\ell_2,m_2]$\\
& $X = [\ell_1,m_1]\cup[\ell_2,m_2]\cup\cdots\cup [\ell_k,m_k]$\\\hline
Shonhiwa \cite{Sh1}, \cite{Sh2} & $X=[1,n]$ with various contraints\\\hline
Ayad and Kihel \cite{AK3}, & Different direction such as congruence  \\
El Bachraoui \cite{Ba4}, \cite{Ba5} & properties, divisor sum types, \\
Tang \cite{Ta}, Toth \cite{To} &combinatorial identities \\ \hline
\end{tabular}
\end{center}
\end{table}
\newpage
\indent In this paper, we give shorter and simpler calculations for these formulas. In Section 3, we show that we need only to derive the formula for $\Phi_k(X, n)$ as the others will follow as a consequence. This will cover the results in \cite{AK1}, \cite{AK2}, \cite{Ba1}, \cite{Ba2}, \cite{Na}, and \cite{NO}. In Section 4, we extend the formulas obtained by Ayad and Kihel \cite{AK1}, \cite{AK2}, by El Bachraoui \cite{Ba3}, \cite{Ba4}, \cite{Ba5}, and by El Bachraoui and Salim \cite{BS}. In Section 5, we show how our method can be used to obtain Shonhiwa's results \cite{Sh1}, \cite{Sh2} in a simpler and shorter way. We conclude this paper by giving a possible research related to the work of Ayad and Kihel \cite{AK3}, Tang \cite{Ta}, and Toth \cite{To}. 
\section{Lemmas}
In this section, we give a formula for the number of terms in an arithmetic progression which are divisible by a fixed positive integer.
\begin{lem}\label{newlemma2}
For integers $d, m\geq 1$ and nonzero integers $a$ and $b$, let $A = \{a, a+b,\ldots, a+(m-1)b\}$ be an $m$-arithmetic progression, $A_d = \{x\in A\;:\;d\mid x\}$, and $k = \gcd (d,b)$. Then 
\begin{itemize}
\item [(i)] If $k\nmid a$, then $|A_d| = 0$. 
\item [(ii)] If $k\mid a$, then $\displaystyle |A_d| = \left\lfloor \frac{mk}{d}\right\rfloor+\varepsilon_d$ where\\
$\displaystyle \varepsilon_d = 
\begin{cases}
1,\quad& \text{if $d\nmid mk$ and $-\frac{a}{k}\left(\frac{b}{k}\right)^{-1}\modu \frac{d}{k}$}\in\left\{0, 1,\ldots, m-1-\left\lfloor \frac{(m-1)k}{d}\right\rfloor\frac{d}{k}\right\}\\
&\text{where $\left(\frac{b}{k}\right)^{-1}$ is the multiplication inverse of $\frac{b}{k}$ modulo $\frac{d}{k}$.}\\
0,\quad& \text{otherwise}
\end{cases}$
\end{itemize}
\end{lem}
\begin{proof}
From the definition of $A$ and $A_d$, we see that $|A_d|$ is equal to the number of $x\in\{0,1,2,\ldots,m-1\}$ such that $a+xb\equiv 0(\modu d)$. So we consider the congruence
\begin{equation}\label{conguenceeq1}
bx \equiv -a(\modu d)
\end{equation}
If $k$ does not divide $a$, then there is no $x$ satisfying (\ref{conguenceeq1}) and thus $|A_d|=0$. This proves (i). Next, we assume that $k\mid a$. Then (\ref{conguenceeq1}) becomes
\begin{equation}\label{conguenceeq2}
\frac{b}{k}x \equiv -\frac{a}{k}\left(\modu \frac{d}{k}\right)
\end{equation}
Since $k=(d,b)$, $\left(\frac{d}{k},\frac{b}{k}\right) = 1$. So (\ref{conguenceeq2}) has a unique solution $\modu \frac{d}{k}$ which is 
\begin{equation}\label{conguenceeq3new}
x \equiv  -\frac{a}{k}\left(\frac{b}{k}\right)^{-1}\left(\modu \frac{d}{k}\right)
\end{equation}
where $\left(\frac{b}{k}\right)^{-1}$ is the multiplication inverse of $\frac{b}{k}$ modulo $\frac{d}{k}$.\\
So we want to count the number of elements in the set $\{0,1,2,\ldots,m-1\}$ which satisfy (\ref{conguenceeq3new}). Each of the following sets contain a unique element satisfying (\ref{conguenceeq3new}) \vspace{0.3cm}

$$
\left\{0,1,\ldots,\frac{d}{k}-1\right\},\left\{\frac{d}{k},\frac{d}{k}+1,\ldots,\frac{2d}{k}-1\right\},\ldots,\left\{\left(\left\lfloor \frac{mk}{d}\right\rfloor-1\right)\frac{d}{k},\ldots,\left\lfloor \frac{mk}{d}\right\rfloor\frac{d}{k}-1\right\}.\vspace{0.3cm}
$$

There are $\left\lfloor \frac{mk}{d}\right\rfloor$ sets. This implies that $|A_d| = \left\lfloor \frac{mk}{d}\right\rfloor+\varepsilon_d$ where $\varepsilon_d=1$ if $\left\lfloor \frac{mk}{d}\right\rfloor\frac{d}{k}-1<m-1$ and the set $\left\{\left\lfloor \frac{mk}{d}\right\rfloor\frac{d}{k},\left\lfloor \frac{mk}{d}\right\rfloor\frac{d}{k}+1,\ldots,m-1\right\}$ contains an element satisfying (\ref{conguenceeq3new}), otherwise $\varepsilon_d=0$. \\
It is easy to see that $\left\lfloor \frac{mk}{d}\right\rfloor\frac{d}{k}-1<m-1$ if and only if $d\nmid mk$, we also see that $\left\lfloor \frac{mk}{d}\right\rfloor\frac{d}{k}$, $\left\lfloor \frac{mk}{d}\right\rfloor\frac{d}{k}+1,\ldots,m-1$ are congruent  to $0, 1, 2,\ldots, m-1-\left\lfloor \frac{(m-1)k}{d}\right\rfloor\frac{d}{k}$ modulo $\frac{d}{k}$, respectively. Hence
$$
\varepsilon_d = 
\begin{cases}
1,\quad& \text{if $d\nmid mk$ and $-\frac{a}{k}\left(\frac{b}{k}\right)^{-1}\modu \frac{d}{k}$}\in\left\{0, 1,\ldots, m-1-\left\lfloor \frac{(m-1)k}{d}\right\rfloor\frac{d}{k}\right\}\\
0,\quad& \text{otherwise}.
\end{cases}
$$
This completes the proof.
\end{proof}
If we consider the case $\gcd(a,b)=1$, we obtain a lemma of Ayad and Kihel as a corollary. We record it in the next lemma.
\begin{lem}\label{lem2}
\cite{AK2} For an integer $d, m\geq 1$, and for nonzero integers $a$ and $b$ with $\gcd(a,b)=1$, let $X = \{a, a+b, \ldots, a+(m-1)b\}$ be an $m$-arithmetic progression, $A_d = \{x\in X\;:\;d\mid x\}$, and $k=\gcd(d,b)$. Then 
\begin{itemize}
\item [i)] If $k\neq 1$, then $|A_d| = 0$.
\item [ii)] If $k=1$, then $|A_d| = \left\lfloor \frac{m}{d}\right\rfloor+\varepsilon_d$ where
$$
\varepsilon_d = 
\begin{cases}
1&\quad\text{if}\;d\nmid m\;\text{and}\;\left(-ab^{-1}\right)\modu d\in\{0,\ldots,m-\left\lfloor \frac{m}{d}\right\rfloor d-1\},\\
&\quad\text{where $b^{-1}$ is the multiplication inverse of $b$ modulo $d$.}\\
0&\quad\text{otherwise}.
\end{cases}
$$ 
\end{itemize}
\end{lem}
\begin{proof}
Since $\gcd(a,b)=1$, we see that $k=1$ if and only if $k\mid a$. So if $k\neq 1$, then $|A_d|=0$ and if $k=1$, then $|A_d| = \left\lfloor \frac{m}{d}\right\rfloor+\varepsilon_d$, by Lemma \ref{newlemma2}. Notice also that $\left\lfloor \frac{m}{d}\right\rfloor = \left\lfloor \frac{m}{d}-\frac{1}{d}\right\rfloor$ if and only if $d\nmid m$. So the conditions determining $\varepsilon_d$ in Lemma \ref{newlemma2} and Lemma \ref{lem2} are the same. This completes the proof.
\end{proof}
The next lemma will be used throughout this paper.
\begin{lem}\label{lem6mu}
\quad$\displaystyle \sum_{d\mid n}\mu(d) = 
\begin{cases}
1\quad\text{if}\;n=1\\
0\quad\text{if}\;n>1.
\end{cases}$
\end{lem}
\begin{proof}
This is a well-known result. For the proof see, for example, (\cite{Ap}, p.25).
\end{proof}
\section{Only One Formula Is Enough}
In this section, we give a simple proof of the formula for $\Phi_k^{(a,b)}(m,n)$ and show that the formulas for $\Phi^{(a,b)}(m,n)$, $f_k^{(a,b)}(m)$, and $f^{(a,b)}(m)$ can be obtained as a consequence. In the notation used in \cite{AK1}, \cite{AK2}, $f^{(a,b)}(m)$, $f_k^{(a,b)}(m)$, $\Phi^{(a,b)}(m,n)$ and $\Phi_k^{(a,b)}(m,n)$ are $f(X)$, $f_k(X)$, $\Phi(X)$, and $\Phi_k(X)$, respectively, where $X = \{a, a+b,\ldots, a+(m-1)b\}$. The following are the results obtained by Ayad and Kihel in \cite{AK1} and \cite{AK2}.
\begin{thm}\label{AK2thm1}
\cite{AK2} For all positive integers $m$, $a$ and $b$, let $f^{(a,b)}(m)$ denote the number of relatively prime subsets of $\{a,a+b,\ldots,a+(m-1)b\}$ and $f_k^{(a,b)}(m)$ denote of the number of relatively prime subsets of $\{a,a+b,\ldots,a+(m-1)b\}$ of cardinality $k$. Suppose that $\gcd(a,b) = 1$. Then 
\begin{align*}
f^{(a,b)}(m) &= \sum_{\substack{d=1\\\gcd(d,b)=1}}^{a+(m-1)b} \mu(d)\left(2^{\left\lfloor m/d\right\rfloor+\varepsilon_d}-1\right)\quad\text{and}\\
f_k^{(a,b)}(m) &= \sum_{\substack{d=1\\\gcd(d,b)=1}}^{a+(m-1)b} \mu(d){\left\lfloor m/d\right\rfloor+\varepsilon_d \choose k}
\end{align*}
where $\varepsilon_d$ is the function defined in Lemma \ref{lem2}.\\
If $\gcd(a,b)\neq 1$, it is easy to see that $f^{(a,b)}(m) = f^{(a,b)}_k(m) = 0$.
\end{thm}
\begin{thm}\label{AK1thm3}
\cite{AK1} For positive integers $m, a,$ and $b$, let $\Phi^{(a,b)}(m,n)$ be the number of nonempty subsets of $\{a+,a+b,\ldots,a+(m-1)b\}$ which are relatively prime to $n$ and let $\Phi_k^{(a,b)}(m,n)$ be the number of such subsets of cardinality $k$. Suppose that $\gcd(a,b) = 1$. Then 
\begin{align*}
\Phi^{(a,b)}(m,n)&=\sum_{\substack{d\mid n\\\gcd(b,d)=1}}\mu(d)\left(2^{\left\lfloor \frac{m}{d}\right\rfloor +\varepsilon_d}-1\right)\quad\text{and}\\
\Phi_k^{(a,b)}(m,n)&=\sum_{\substack{d\mid n\\\gcd(b,d)=1}}\mu(d){\left\lfloor \frac{m}{d}\right\rfloor +\varepsilon_d\choose k}
\end{align*}
where $\varepsilon_d$ is the function defined in Lemma \ref{lem2}
\end{thm}
\begin{cor}\label{cor4}
(\cite{AK1}, \cite{AK2}) The formulas for $f(m,k)$, $f_k(m,\ell)$, $\Phi(m,\ell)$, $\Phi_k(m,\ell)$ $\Phi([1,m],n)$, and $\Phi_k([1,m],n)$ obtained in \cite{Na}, \cite{Ba1}, \cite{NO}, and \cite{Ba2} are consequences of Theorem \ref{AK2thm1} and Theorem \ref{AK1thm3}.
\end{cor}
Now we will give a proof of the formula for $\Phi_k^{(a,b)}(m,n)$ and show that the other formulas follow as a consequence.
\begin{proof}
Let $X = \{a, a+b, \ldots, a+(m-1)b\}$ and for each $d$, we let $A_d = \{x\in X\;:\;d\mid x\}$. Then the definition of $\Phi_k^{(a,b)}(m,n)$ can be written as 
$$
\Phi_k^{(a,b)}(m,n) = \sum_{\substack{\emptyset \neq A \subseteq X\\|A| = k\\\gcd(A,n)=1}} 1.
$$
Now we capture the condition $\gcd(A,n)=1$ by Lemma \ref{lem6mu} and write
$$
\Phi_k^{(a,b)}(m,n) = \sum_{\substack{\emptyset \neq A \subseteq X\\|A| = k}}\sum_{d\mid \gcd(A,n)} \mu(d) = \sum_{\substack{\emptyset \neq A \subseteq X\\|A| = k}}\sum_{\substack{d\mid \gcd(A)\\d\mid n}} \mu(d) 
$$
Changing the order of summation, the above sum becomes
$$
\sum_{d\mid n} \mu(d) \sum_{\substack{\emptyset \neq A \subseteq X\\|A| = k\\ d\mid \gcd(A)}}1.
$$
$d\mid\gcd(A)$ if and only if $d$ divides all elements of $A$. So the condition $\emptyset \neq A \subseteq X$ and $d\mid \gcd(A)$ is equivalent to $\emptyset \neq A \subseteq A_d$. Hence the above sum is equal to 
$$
\sum_{d\mid n} \mu(d) \sum_{\substack{\emptyset \neq A \subseteq A_d\\|A| = k}}1.
$$
By Lemma \ref{lem2}, $|A_d|=0$ if $\gcd(d,b)\neq 1$. So for nonzero contribution, we can restrict our attention to the case $\gcd(d,b)=1$. Therefore the above sum is equal to 
\begin{equation}\label{eq67}
\sum_{\substack{d\mid n\\\gcd(d,b)=1}} \mu(d) \sum_{\substack{\emptyset \neq A \subseteq A_d\\|A| = k}}1 = \sum_{\substack{d\mid n\\\gcd(d,b)=1}} \mu(d){|A_d|\choose k}.
\end{equation}
Applying Lemma \ref{lem2} again, we substitute $|A_d| = \left\lfloor \frac{m}{d}\right\rfloor+\varepsilon_d$ to get
$$
\Phi_k^{(a,b)}(m,n) = \sum_{\substack{d\mid n\\\gcd(d,b)=1}} \mu(d){\left\lfloor m/d\right\rfloor+\varepsilon_d\choose k}.
$$
To obtain the formula of $\Phi^{(a,b)}(m,n)$, we use the well-known identity that
$$
\sum_{k=1}^n {n\choose k} = 2^n-1.
$$
This can also be written as $\displaystyle \sum_{k=1}^\infty {n\choose k} = 2^n-1$ since ${n\choose k}=0$ when $k>n$. Now by the definition of $\Phi^{(a,b)}(m,n)$, we have
\begin{align*}
\Phi^{(a,b)}(m,n) &= \sum_{k=1}^m\Phi_k^{(a,b)}(m,n) = \sum_{k=1}^m\sum_{\substack{d\mid n\\\gcd(d,b)=1}} \mu(d) {|A_d|\choose k}\\
&= \sum_{\substack{d\mid n\\\gcd(d,b)=1}} \mu(d) \sum_{k=1}^m{|A_d|\choose k}\\
&= \sum_{\substack{d\mid n\\\gcd(d,b)=1}} \mu(d)\left(2^{|A_d|}-1\right) \\
&= \sum_{\substack{d\mid n\\\gcd(d,b)=1}} \mu(d)\left(2^{\left\lfloor m/d\right\rfloor+\varepsilon_d}-1\right),\quad\text{as required.}
\end{align*}
Next, we put $n = (a+(m-1)b)!$ in the formula of $\Phi_k^{(a,b)}(m,n)$. For a nonempty subset $A$ of $\{a, a+b,\ldots,a+(m-1)b\}$, we have $\gcd(A,n)=1$ if and only if $\gcd(A)=1$. Therefore by the definition of $\Phi_k^{(a,b)}(m,n)$ and $f_k^{(a,b)}(m)$, we have
\begin{equation}\label{eq1.1}
\Phi_k^{(a,b)}(m,n) = f_k^{(a,b)}(m).
\end{equation}
On the other hand, we have from (\ref{eq67}) that
$$
\Phi_k^{(a,b)}(m,n) = \sum_{\substack{d\mid n\\\gcd(d,b)=1}} \mu(d) {|A_d|\choose k}.
$$
Notice that $d = 1, 2, \ldots, a+(m-1)b$ are divisors of $n$ and if $d> a+(m-1)b$, then $d$ is larger than all elements of $A$ and thus $|A_d|=0$. Therefore the above sum is 
\begin{equation}\label{eq1.2}
\Phi_k^{(a,b)}(m,n) = \sum_{\substack{d=1\\\gcd(d,b)=1}}^{a+(m-1)b} \mu(d){|A_d|\choose k}
\end{equation}
From (\ref{eq1.1}), (\ref{eq1.2}) and Lemma \ref{lem2}, we obtain
$$
f_k^{(a,b)}(m) = \sum_{\substack{d=1\\\gcd(d,b)=1}}^{a+(m-1)b} \mu(d){|A_d|\choose k} = \sum_{\substack{d=1\\\gcd(d,b)=1}}^{a+(m-1)b} \mu(d){\left\lfloor m/d\right\rfloor+\varepsilon_d \choose k}.
$$ 
Similar to the proof of $\Phi^{(a,b)}(m,n)$, we sum $f_k^{(a,b)}(m)$ over all $k$ to get $f^{(a,b)}(m)$. This completes the proof.
\end{proof} 
\begin{rem}
As noted by Ayad and Kihel in \cite{AK2} that we can easily deduced from Theorem \ref{AK1thm3} the results for the case when $a$ and $b$ are integers not necessary positive, or the case $((a,b),n)\neq 1$ or $(a,b)\neq 1$ but $((a,b),n) = 1$. For the details, see Remark 11 and Remark 12 in \cite{AK2}. Combining this with Corollary \ref{cor4}, we see that we cover the results given by Ayad and Kihel (\cite{AK1}, \cite{AK2}), El Bachraoui (\cite{Ba1}, \cite{Ba2}), Nathanson (\cite{Na}) and Nathanson and Orosz (\cite{NO}).
\end{rem}
\section{Extending the formulas to finite union of arithmetic progressions}

In this section, we will give formulas for $f(X), f_k(X), \Phi(X,n), \Phi_k(X,n)$ when
\begin{align*}
X &= [a_1,b_1]\cup[a_2,b_2]\cup\ldots\cup[a_\ell,b_\ell]\quad\text{or}\\ 
X &= \{a_1, a_1+b_1,\ldots,a_1+(m_1-1)b_1\}\cup\{a_2, a_2+b_2,\ldots,a_2+(m_2-1)b_2\}\\
&\quad\cup\ldots\cup\{a_\ell, a_\ell+b_\ell,\ldots,a_\ell+(m_\ell-1)b_\ell\}.
\end{align*}
\indent Considering our method carefully, we see that it can be applied in any situation where the number of elements divisible by a fixed positive integer can be calculated. We illustrate this idea explicitly below.\\
Let $X$ be a nonempty finite subset of integers and for each $d$, let $X_d = \{x\in X\;:\;d\mid x\}$. 
By applying Lemma \ref{lem6mu} and changing the order of summation, we have
\begin{align}\label{eq1sec5exten}
\Phi_k(X,n) &= \sum_{\substack{\emptyset \neq A \subseteq X\\|A| = k\\\gcd(A,n)=1}} 1 = \sum_{\substack{\emptyset \neq A \subseteq X\\|A| = k}}\sum_{d\mid \gcd(A,n)} \mu(d)\notag\\
&=\sum_{d\mid n}\mu(d)\sum_{\substack{\emptyset \neq A \subseteq X\\|A| = k\\d\mid \gcd(A)}} 1 \notag\\
&= \sum_{d\mid n}\mu(d)\sum_{\substack{\emptyset \neq A \subseteq X_d\\|A| = k}} 1\notag\\
& =\sum_{d\mid n}\mu(d){|X_d|\choose k}.
\end{align}
Summing over all $k$, we see that
\begin{equation}\label{eq2sec5exten}
\Phi(X,n) = \sum_{k=1}^{|X|}\Phi_k(X,n) = \sum_{d\mid n}\mu(d)\sum_{k=1}^{|X|}{|X_d|\choose k} = \sum_{d\mid n}\mu(d)\left(2^{|X_d|}-1\right).
\end{equation}
Again, applying Lemma \ref{lem6mu} and changing the order of summation, we have
\begin{align}\label{eq3sec5exten}
f_k(X)  &=\sum_{\substack{\emptyset \neq A \subseteq X\\|A| = k}} \sum_{d\mid \gcd(A)}\mu(d) = \sum_{d=1}^{\max X}\mu(d)\sum_{\substack{\emptyset \neq A \subseteq X\\|A| = k\\d\mid\gcd (A)}}1\notag\\
&= \sum_{d=1}^{\max X}\mu(d)\sum_{\substack{\emptyset \neq A \subseteq X_d\\|A| = k}} 1 = \sum_{d=1}^{\max X}\mu(d){|X_d|\choose k}.
\end{align}
Summing $f_k(X)$ over all $k$, we obtain
\begin{equation}\label{eq4sec5exten}
f(X) = \sum_{d=1}^{\max X}\mu(d)\left(2^{|X_d|}-1\right)
\end{equation}
\begin{rem}
1) If $n>1$, by Lemma \ref{lem6mu}, the formula in (\ref{eq2sec5exten}) can be reduced to
$$
\Phi(X,n) = \sum_{d\mid n}\mu(d)2^{|X_d|}.
$$ 
2) From (\ref{eq1sec5exten}), (\ref{eq2sec5exten}), (\ref{eq3sec5exten}), and (\ref{eq4sec5exten}), we see that explicit formulas for $\Phi_k(X,n), \Phi(X,n),$ $f_k(X)$ and $f(X)$ can be obtained whenever we can compute $|X_d|$ for all $d$.
\end{rem}
With equations (\ref{eq1sec5exten}), (\ref{eq2sec5exten}), (\ref{eq3sec5exten}), and (\ref{eq4sec5exten}), we obtain the following theorem.
\begin{thm}\label{thm1sec5extend}
Let $X = [a_1,b_1]\cup[a_2,b_2]\cup\ldots\cup[a_\ell,b_\ell]$ where $a_1\leq b_1<a_2\leq b_2<\ldots<a_\ell\leq b_\ell$. Assume that $[a_i,b_i]\cap[a_j,b_j]=\emptyset$ for $i\neq j$. Then
\begin{itemize}
\item [(i)] $\displaystyle \Phi_k(X,n) = \sum_{d\mid n}\mu(d){\sum_{i=1}^\ell\left\lfloor \frac{b_i}{d}\right\rfloor-\left\lfloor \frac{a_i-1}{d}\right\rfloor\choose k}$
\item [(ii)] $\displaystyle \Phi(X,n) = \sum_{d\mid n}\mu(d)\left(2^{\sum_{i=1}^\ell\left\lfloor \frac{b_i}{d}\right\rfloor-\left\lfloor \frac{a_i-1}{d}\right\rfloor}-1\right)$
\item [(iii)] $\displaystyle f_k(X) = \sum_{d=1}^{b_\ell}\mu(d){\sum_{i=1}^\ell\left\lfloor \frac{b_i}{d}\right\rfloor-\left\lfloor \frac{a_i-1}{d}\right\rfloor\choose k}$
\item [(iv)] $\displaystyle f(X) = \sum_{d=1}^{b_\ell}\mu(d)\left(2^{\sum_{i=1}^\ell\left\lfloor \frac{b_i}{d}\right\rfloor-\left\lfloor \frac{a_i-1}{d}\right\rfloor}-1\right)$
\end{itemize}
\end{thm}
\begin{proof}
The number of integers $x\in [1,n]$ divisible by $d$ is equal to $\left\lfloor \frac{n}{d}\right\rfloor$. So the number of integers $x\in [a,b]$ such that $d\mid x$ is equal to $\left\lfloor \frac{b}{d}\right\rfloor-\left\lfloor \frac{a-1}{d}\right\rfloor$. This implies that 
$$
|X_d| = \sum_{i=1}^\ell\left\lfloor \frac{b_i}{d}\right\rfloor-\left\lfloor \frac{a_i-1}{d}\right\rfloor.
$$
Substituting this in (\ref{eq1sec5exten}), (\ref{eq2sec5exten}), (\ref{eq3sec5exten}), and (\ref{eq4sec5exten}), we obtain the desired result.
\end{proof}
Note that the formulas in Theorem \ref{thm1sec5extend} are also obtained by El Bachraoui \cite{Ba4}, \cite{Ba5} in a different form but his proof does not seem to be applicable in more general situations such as \cite{Sh1}, \cite{Sh2}. However, our method still works well in this case (see section 5). 
\begin{thm}\label{thm2sec5extend}
Let $\displaystyle X = \bigcup_{i=1}^\ell I_i$ where $I_i = \{a_i,a_i+b_i,\ldots, a_i+(m_i-1)b_i\}$ be an $m_i$-arithmetic progression. Assume that $(a_i,b_i)=1$ for all $i$ and $I_i\cap I_j = \emptyset$ for all $i\neq j$. Then 
\begin{itemize}
\item [(i)] $\displaystyle \Phi_k(X,n) = \sum_{d\mid n}\mu(d){\sum_{i=1}^\ell|I_{id}|\choose k}$
\item [(ii)] $\displaystyle \Phi(X,n) = \sum_{d\mid n}\mu(d)\left(2^{\sum_{i=1}^\ell|I_{id}|}-1\right)$
\item [(iii)] $\displaystyle f_k(X) = \sum_{d=1}^{\max X}\mu(d){\sum_{i=1}^\ell|I_{id}|\choose k}$
\item [(iv)] $\displaystyle f(X) = \sum_{d=1}^{\max X}\mu(d)\left(2^{\sum_{i=1}^\ell|I_{id}|}-1\right)$
\end{itemize}
where for each $i$ and $d$, $|I_{id}|=0$ if $(d,b_i)\neq 1$ and if $(d, b_i)=1$, then $|I_{id}| = \left\lfloor \frac{m_i}{d}\right\rfloor+\varepsilon_{id}$ where
$$
\varepsilon_{id} = 
\begin{cases}
1&\quad\text{if}\;d\nmid m_i\;\text{and}\;-a_ib_i^{-1}\modu d\in\{0,\ldots,m_i-1-\left\lfloor \frac{m_i}{d}\right\rfloor d\},\\
&\quad\text{where $b_i^{-1}$ is the multiplication inverse of $b_i$ modulo $d$.}\\
0&\quad\text{otherwise}.
\end{cases}
$$
\end{thm}
\begin{proof}
We have $|X_d| = \displaystyle \sum_{i=1}^\ell|I_{id}|$ and $|I_{id}|$ can be obtained by Lemma \ref{lem2}. This completes the proof.
\end{proof}
\begin{rem}
1) If $(a_i,b_i)>1$ for some $i$, we can apply Lemma \ref{newlemma2} to obtain the corresponding result to Theorem \ref{thm2sec5extend}.\\
2) Theorem \ref{thm2sec5extend} extends Ayad and Kihel's results \cite{AK1}, \cite{AK2} to the case of finite union of arithmetic progressions. Replacing $\ell=1$ and $X=\{a,a+b,\ldots,a+(m-1)b\}$, we obtain their result in \cite{AK1} and \cite{AK2}. 
\end{rem}
\section{Cover Shonhiwa's theorems}
Shonhiwa considers the case $X=[1,n]$ with various constraints. He \cite{Sh1}, \cite{Sh2} uses the M$\ddot{\text{o}}$bius inversion formula, the inclusion-exclusion principle, generating functions, and standard formulas in enumerative combinatorics. In this section, we illustrate again how our method can be used to obtain Shonhiwa's results in a faster and simpler way. So let us recall his theorems in \cite{Sh1}, \cite{Sh2}.
\begin{thm}
(\cite{Sh1}, \cite{Sh2}) Let
\begin{align*}
(i)\; S_k^m(n)\;&=\;\sum_{\substack{1\leq a_1,a_2,\ldots, a_k\leq n\\(a_1,a_2,\ldots,a_k,m)=1}}1;\quad\forall n\geq k\geq 1, m\geq 1\\
(ii)\; G_k(n)\;&=\;\sum_{\substack{1\leq a_1,a_2,\ldots, a_k\leq n\\(a_1,a_2,\ldots,a_k)=1}}1;\quad\forall n\geq k\geq 1\\
(iii)\; L_k^m(n)\;&=\;\sum_{\substack{1\leq a_1\leq a_2\leq \ldots\leq a_k\leq n\\(a_1,a_2,\ldots,a_k,m)=1}}1;\quad\forall n\geq k\geq 1, m\geq 1\\
(iv)\; H_k(n)\;&=\;\sum_{\substack{1\leq a_1\leq a_2\leq \ldots\leq a_k\leq n\\(a_1,a_2,\ldots,a_k)=1}}1;\quad\forall n\geq k\geq 1\\
(v)\; T_k^m(n)\;&=\;\sum_{\substack{1\leq a_1< a_2< \ldots< a_k\leq n\\(a_1,a_2,\ldots,a_k,m)=1}}1;\quad\forall n\geq k\geq 1, m\geq 1
\end{align*}
Then
\begin{align*}
S_k^m(n)\;&=\;\sum_{d\mid m}\mu(d)\left\lfloor \frac{n}{d}\right\rfloor^k,\\
G_k(n)\;&=\;\sum_{d\leq n}\mu(d)\left\lfloor \frac{n}{d}\right\rfloor^k,\\
L_k^m(n)\;&=\;\sum_{d\mid m}\mu(d){\left\lfloor \frac{n}{d}\right\rfloor+k-1\choose k}\\
H_k(n)\;&=\;\sum_{d\leq n}\mu(d){\left\lfloor \frac{n}{d}\right\rfloor+k-1\choose k}\\
T_k^m(n)\;&=\;\sum_{d\mid m}\mu(d){\left\lfloor \frac{n}{d}\right\rfloor\choose k}.
\end{align*}
\end{thm}
\begin{proof}
Throughout the proof, we let 
\begin{align*}
A &= \{1, 2,\ldots, n\}\quad\text{and} \\
A_d &= \{a\in A\;:\;d\mid a\} = \{d, 2d, 3d,\ldots, \left\lfloor \frac{n}{d}\right\rfloor d\}.
\end{align*}
For (i), we apply Lemma \ref{lem6mu} and change the order of summation to obtain 
\begin{align*}
S_k^m(n)\;&=\;\sum_{1\leq a_1,a_2,\ldots, a_k\leq n}\sum_{d\mid (a_1,a_2,\ldots, a_k, m)} \mu(d)\\
&=\;\sum_{d\mid m}\mu(d)\sum_{\substack{1\leq a_1,a_2,\ldots, a_k\leq n\\d\mid a_1, d\mid a_2,\ldots, d\mid a_k}}1.
\end{align*}
The condition $1\leq a\leq n$ and $d\mid a$ is the same as $a\in A_d$. Therefore the number of choices for $a$ is $|A_d| = \left\lfloor \frac{n}{d}\right\rfloor$. So we have
$$
S_k^m(n)\;=\;\sum_{d\mid m}\mu(d)\sum_{a_1,a_2,\ldots, a_k\in A_d}1\;=\;\sum_{d\mid m}\mu(d)|A_d|^k\;=\;\sum_{d\mid m}\mu(d)\left\lfloor \frac{n}{d}\right\rfloor^k.
$$ 
\noindent For (ii) we put $m = n!$ in $S_k^m(n)$ and argue as in the proof of $\Phi^{(a,b)}(m,n)$ in the previous section. We see that 
$$
G_k(n)\;=\;S_k^m(n)\;=\;\sum_{d\mid m}\mu(d)\left\lfloor \frac{n}{d}\right\rfloor^k\;=\;\sum_{d\mid n!}\mu(d)\left\lfloor \frac{n}{d}\right\rfloor^k\;=\;\sum_{d=1}^n\mu(d)\left\lfloor \frac{n}{d}\right\rfloor^k.
$$
Before giving the proof of (iii), let us recall an elementary formula in enumeration. The number of ways to select $k$ objects from $n$ different objects with repetition allowed is equal to (\cite{CK}, p.47)
\begin{equation}\label{waystoselect}
{k+n-1\choose k}.
\end{equation}
Now similar to (i), we apply Lemma \ref{lem6mu} and change the order of summation to obtain 
$$
L_k^m(n) = \sum_{d\mid m}\mu(d)\sum_{\substack{1\leq a_1\leq a_2\leq \ldots\leq a_k\leq n\\d\mid a_1, d\mid a_2,\ldots, d\mid a_k}}1 = \sum_{d\mid m}\mu(d)\sum_{\substack{a_1,a_2, \ldots, a_k\in A_d\\a_1\leq a_2\leq \ldots\leq a_k}}1.
$$
The inner sum is equal to the number of ways to select $k$ objects from $|A_d| = \left\lfloor \frac{n}{d}\right\rfloor$ different objects with repetition allowed. So it is equal to ${k+\left\lfloor \frac{n}{d}\right\rfloor-1\choose k}$ by (\ref{waystoselect}). Thus 
\begin{equation}\label{eq8lkm}
L_k^m(n) = \sum_{d\mid m}\mu(d){k+\left\lfloor \frac{n}{d}\right\rfloor-1\choose k}.
\end{equation}
\noindent For (iv), we put $m=n!$ in (\ref{eq8lkm}) and argue as before to get $H_k(n)$.\\
\noindent Similar to (i) and (iii), we have
$$
T_k^m(n) = \sum_{d\mid m}\mu(d)\sum_{\substack{1\leq a_1< a_2<\ldots< a_k\leq n\\d\mid a_1, d\mid a_2,\ldots, d\mid a_k}}1= \sum_{d\mid m}\mu(d)\sum_{\substack{a_1,a_2, \ldots, a_k\in A_d\\a_1<a_2<\ldots< a_k}}1.
$$
The inner sum is the same as $k$-combinations of $|A_d| = \left\lfloor \frac{n}{d}\right\rfloor$ distinct objects, which equals ${\left\lfloor \frac{n}{d}\right\rfloor\choose k}$. This gives (v). Hence the proof is complete.
\end{proof}
\noindent\textbf{Possible research questions}\vspace{0.2cm}\\

There are some interesting problems that can be investigated further. For example, combinatorial identities given by Bachraoui and Salim in \cite{Ba4}, divisor-type functions by Toth \cite{To}, arithmetic properties of the sequence $f(n)=f([1,n])$, and $\Phi(n) = \Phi([1,n])$ by Ayad and Kihel \cite{AK3} and Min Tang \cite{Ta} can be developed more. The possible research questions are the following:
\begin{itemize}
\item [(i)] Is $f(n)$ a cube, a fourth power, or a perfect power for some $n \geq 2$? (Ayad and Kihel \cite{AK3} proved that $f(n)$ is never a square for $n\geq 2$.)
\item [(ii)] What are congruence relations for $f(n), f_k(n), \Phi(n)$, and $\Phi_k(n)$? (Ayad and Kihel \cite{AK3} showed that $\Phi(n) \equiv 0(\modu 3)$ for $n\geq 3$.)
\item [(iii)] Does the sequence $f(n)$ contain infinitely many primes?
\item [(iv)] What are the solutions of $f(n) = \Phi(m)$ ?
\item [(v)] What are the properties of the divisor sums $\sum_{d\mid n}f(d)$ and $\sum_{d\mid n}\Phi(d)$? Is there a combinatorial interpretation associated with these sums?
\item [(vi)] How closed are the sequences $f(n)$, $\Phi(n)$, $\sum_{d\mid n}f(d)$, and $\sum_{d\mid n}\Phi(d)$ to the sequence $2^n$? Which one is nearer to $2^n$ than the others? How closed are the sequence $\sum_{n\leq N}f(n)$, $\sum_{n\leq N}\Phi(n)$, $\sum_{n\leq N}\sum_{d\mid n}f(d)$, and $\sum_{n\leq N}\sum_{d\mid n}\Phi(d)$ to the expected value $2^{N+1}?$ Which one is nearer to $2^{N+1}$ than the others?
\end{itemize}

\noindent \textbf{Acknowledgement}: The author's research is financially supported by Faculty of Science, Silpakorn University, Thailand, contract number RGP 2555-07. He also wishes to thank the anonymous referee and the editor (Professor Bruce Landman) for careful reading and pointing him some typographical and grammatical errors.

\end{document}